\documentclass[12pt,reqno]{amsart}

\usepackage[hmargin=2cm,bmargin=2cm,tmargin=2.5cm]{geometry}
\usepackage{graphicx}
\usepackage{subfigure}
\usepackage{color}
\usepackage{enumerate}
\usepackage[english]{babel}
\usepackage{mathtools}
\usepackage[utf8]{inputenc}
\usepackage[T1]{fontenc}
\usepackage[b]{esvect}
\usepackage{parskip}
\usepackage{bm}
\usepackage{array}
\usepackage{arydshln}
\usepackage{multirow}
\usepackage{bigstrut}
\usepackage{bbm}
\usepackage{esdiff}
\usepackage{relsize}
\usepackage{hyperref}
\usepackage{cleveref}
\usepackage{enumitem}

\newtheorem{thm}{Theorem}[section]
\newtheorem{lem}[thm]{Lemma}
\newtheorem{prop}[thm]{Proposition}

\theoremstyle{definition}
\newtheorem{defn}{Definition}

\newtheorem{remark}[thm]{Remark}

\newcommand{\cA}{\mathcal{A}}

\newcommand{\cD}{\mathcal{D}}
\newcommand{\cE}{\mathcal{E}}

\newcommand{\cM}{\mathcal{M}}

\newcommand{\cU}{\mathcal{U}}
\newcommand{\cV}{\mathcal{V}}
\newcommand{\cW}{\mathcal{W}}

\newcommand{\bC}{\mathbb{C}}

\newcommand{\bR}{\mathbb{R}}

\newcommand{\bT}{\mathbb{T}}

\newcommand{\bZ}{\mathbb{Z}}

\newcommand{\tr}{\mathrm{tr}}

\def\up#1{^{(#1)}}

\newcommand{\vt}{{\mathfrak{v}}}

\usepackage{relsize}
\usepackage{environ}
\NewEnviron{myequation}{%
	\begin{equation*}
	\scalebox{.825}{$\BODY$}
	\end{equation*}
}
\NewEnviron{myequation1}{%
	\begin{equation*}
	\scalebox{.925}{$\BODY$}
	\end{equation*}
}
\NewEnviron{myequation2}{%
	\begin{equation*}
	\scalebox{.925}{$\BODY$}
	\end{equation*}
}

\author{Mahmood Ettehad}
\author{Burak Hat\.{i}no\u{g}lu}
\title[On the spectra of periodic discontinuous quantum graphs]{On the spectra of periodic discontinuous quantum graphs}

\date{}

\begin{document}
	\maketitle
	
	\begin{abstract}
	We consider periodic Schr\"{o}dinger operators on the hexagonal lattice with self-adjoint vertex conditions that allow discontinuity and concentrated mass at the vertices. This model generalizes the periodic Schr\"{o}dinger operator on the hexagonal lattice with Neumann vertex conditions, known as the graphene Hamiltonian quantum graph \cite{KP07}. After formulating the corresponding Hamiltonian as a quantum graph and introducing the self adjoint vertex conditions, we provide its dispersion relation, spectrum, eigenvalues, and Dirac points. 

    We also give explicit formulations of our results for the corresponding free operator (zero potential) and show that Borg's theorem is not valid for the Hamiltonians we study, that is, non-degenerate spectral gaps exist for the free operator with our vertex conditions.
    
\end{abstract}

\section{Introduction}
\label{Section1}

A quantum graph consists of a metric graph, a differential operator, and vertex conditions \cite{BK13}. The theory of quantum graphs represents a significant area of modern mathematical physics that intersects spectral theory, ordinary and partial differential equations, and discrete mathematics. These Hamiltonians model physical phenomena constrained to a metric graph using differential equations defined on the edges with conditions on the vertices \cite{BK13,KP07}. 

In this paper, we consider the Hamiltonian $\cA$ that consists of the Schr\"{o}dinger operator 
        \begin{equation}
	   -a\frac{d^2}{dx^2} + q(x)
	\end{equation}
with $a > 0$, a periodic real symmetric potential $q(x)$ on the hexagonal lattice and self-adjoint vertex conditions \eqref{eq:FBVertexCondEta1} and \eqref{eq:FBVertexCondEta2} that allow discontinuities and concentrated mass at the vertices. The vertex conditions are determined by semi-rigidity ($\kappa^{-1}$) and concentrated mass ($m$) parameters, see Definition \ref{def:semiRigid} and conditions \eqref{eq:FBVertexCondEta1} and \eqref{eq:FBVertexCondEta2}. The semi-rigidity parameter provides controlled discontinuity of solutions in terms of their derivatives as defined in \eqref{eq:semiRigid_displacement}. When the parameters $\kappa^{-1}$ and $m$ are zero, the vertex conditions become Neumann conditions, so our model is a generalization of the graphene Hamiltonian; see (2.4) in \cite{KP07}. The graphene Hamiltonian, i.e. periodic second order Schr\"{o}dinger operator on the hexagonal lattice with Neumann vertex conditions, is a fundamental operator in the quantum graphs literature that models graphene, and its spectral properties were extensively studied by Kuchment and Post \cite{KP07}.

Our main focus is the spectral properties of the Hamiltonian $\cA$. Using Floquet-Bloch theory, one obtains the spectral properties of a periodic operator from the spectral properties of the same operator defined on a compact set (fundamental domain) with additional vertex conditions (cyclic or Floquet-Bloch conditions) depending on quasi-momentum parameters $\theta_1,\theta_2$ \cite{K93,K16}, see Figure \ref{fig:periodicGeom}. The graph of the multi-valued function mapping quasi-momenta $\Theta:=\{\theta_1,\theta_2\}$ to the spectrum of the corresponding Bloch Hamiltonian $\cA^{\Theta}$ is called the dispersion relation, or the Bloch variety of $\cA$, see Subsection \ref{subsecQuasimomentum} for a discussion of $\cA^{\Theta}$ and its connection with $\cA$. In Proposition \ref{prop2}, we obtain the dispersion relation of $\cA$ and hence characterize its spectrum. In Lemma \ref{sigmaDLem}, we provide eigenvalues of $\cA$ in the cases of rigid and semi-rigid vertices.

The Hamiltonian $\cA$ is a self-adjoint operator, hence its spectrum consists of absolutely continuous, pure point, and singular continuous spectra. In Theorem \ref{thm:DispThm}, we prove a spectral description of $\cA$, including singular continuous, absolutely continuous, and pure point spectra. If two sheets of the dispersion relation of an operator touch at a point and form a conical singularity, such a point is called a Dirac point. In Proposition \ref{cor:DiracPoints} and the following remarks, we discuss Dirac points depending on the existence of mass and semi-rigidity at the vertices.

Secondly, we consider the free operator $\cA_0$ (i.e. $\cA$ with $q\equiv 0$). In Theorem \ref{thm:DispThmfree} and Proposition \ref{cor:DiracPointsFree}, we obtain representations of the dispersion relation, the spectrum, the eigenvalues and the Dirac points of $\cA_0$ in terms of the semi-rigidity ($\kappa^{-1}$) and concentrated mass ($m$) parameters, and the fundamental solutions $\cos(\sqrt{\lambda/a})$ and $\sin(\sqrt{\lambda/a})/\sqrt{\lambda/a}$, where $\lambda$ is the eigenvalue parameter.

The classical Borg's theorem says that the one dimensional periodic Schr\"{o}dinger operator has no spectral gaps if and only if the potential $q$ is constant \cite{B46}. Using Theorem \ref{thm:DispThmfree}, in Remark \ref{BorgFreeRemark} we discuss that the classical Borg's theorem is not valid for the Hamiltonian $\cA$, i.e. the free operator $\cA_0$ has non-degenerate spectral gaps if one of the of the semi-rigidity ($\kappa^{-1}$) or concentrated mass ($m$) parameters is nonzero, see Figure \ref{fig:DeltaFunc}. Remark \ref{EigenvaluesBandenpointsFree} addresses that any eigenvalue of $\cA_0$ is an embedded eigenvalue and an endpoint of a spectral band of $\cA_0$. Lemma \ref{SigmaEtaD1free} characterizes the end points of the spectrum of $\cA_0$ in terms of the parameters $a$, $\kappa^{-1}$, $m$ and the fundamental solutions $\cos(\sqrt{\lambda/a})$ and $\sin(\sqrt{\lambda/a})/\sqrt{\lambda/a}$.

The paper is structured as follows. In Section \ref{Section2}, we provide a formulation of the Hamiltonian $\cA$ by introducing semi-rigidity of vertices, variational and differential formulations, and then interpretation of the self-adjoint operator $\cA$ with its vertex conditions. In Section \ref{Section3}, we obtain spectral properties of $\cA$, namely descriptions of the dispersion relation, absolutely continuous, singular continuous, and pure point spectra, and the Dirac points. In Section \ref{Section4}, we do a detailed study of the spectral properties of the free operator $\cA_0$ and provide the results discussed above.

\begin{figure}[h]
	\centering
	\includegraphics[width=0.35\textwidth]{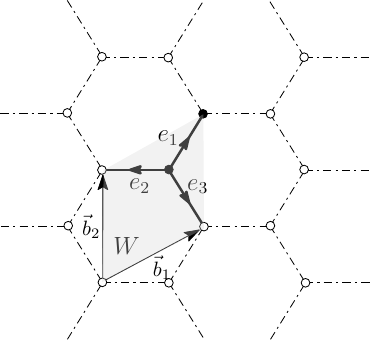}
	\caption{The hexagonal lattice $G$ and a fundamental domain $W$
		together with its set of vertices $V (W) = \{\vt_1, \vt_2\}$ and set of edges $E(W) = \{e_1, e_2,e_3\}$. The vertex set $V(W)$ of the fundamental domain $W$ consists of the endpoints of edge $e_1$.}
	\label{fig:periodicGeom}
\end{figure}

	\section{Formulation of the Hamiltonian $\cA$}
	\label{Section2}
	
	\subsection{Displacement Hamiltonian}
	%Here we calculate the dispersion relation of field $u(x)$ on periodic infinite line. 
    Hamiltonian $\cA$ is similar to Schr\"{o}dinger operator except that jumps and concentrated masses are allowed at the vertices. In this paper, an edge $e$ adjacent to $\vt$ will be denoted by $e \sim \vt$, and, moreover, we set  
	\begin{equation}
		\label{eq:vertexLim}
		u_e(\vt) := \lim_{x\to\vt} u_{e}(x)
	\end{equation}  
     and then introduce semi-rigid vertices. Semi-rigid vertices are introduced to allow controlled discontinuity at the vertices and are motivated by the semi-rigid joints from beam models. In this paper, our goal is to consider the discontinuity and concentrated mass conditions at the vertices for a model that is well studied in the quantum graph literature \cite{BK13,KP07}. The Hamiltonian we consider in this paper is not a beam Hamiltonian, but we hope that our results can be carried out to more complicated quantum graphs that model elastic beam frames (which are given by vector-valued fourth-order operators \cite{BE21}) and other physical phenomena in future research. Therefore, even though we study a second-order Hamiltonian, we use some terminology from beam theory.   
     
     Now, let us consider how to generalize the rigid-joint model to semi-rigid. We consider the geometric graph $G = (\cV, \cE)$, consisting of the set of vertices $\cV$ and the set of edges $\cE$. Each edge $e$ provides the following information: origin and terminus vertices $\vt^{o}_e, \vt^{t}_{e} \in \cV$, and the local basis $\{\vv{i_e},\vv{j_e}\}$. We will use \textit{the sign indicator} $s^{\vt}_e$ which is defined as $1$ if $\vt = \vt^{o}_e$, $-1$ if $\vt = \vt^{t}_e$ and $0$ otherwise. This sign convention is consistent with the sign of an outward normal derivative vector of an external surface applied in the continuum mechanics literature, e.g., see \cite{ER15}. Now we are ready to define semi-rigid vertices.
 
 \begin{defn}
		\label{def:semiRigid}
		A vertex $\vt$ is called
		\textit{semi-rigid}, if for each edge $e \sim \vt$ we have
			\label{eq:semirigidVertexDefn}
			\begin{gather}
				\label{eq:semiRigid_displacement}
				u_e(\vt) + s_e^\vt a_e  \kappa^{-1} u_e'(\vt)  = u_\vt^{\hspace{0.5mm} \circ},  
			\end{gather}
   where $\kappa \in (0,\infty)$ is the rigidity parameter and $u_\vt^{\hspace{0.5mm} \circ}$ is independent of the edges adjacent to $\vt$. Note that $\kappa \rightarrow \infty$ provides rigid vertices that satisfy 
   \begin{equation}
        u_e(\vt)  = u_\vt^{\hspace{0.5mm} \circ}.
   \end{equation}
	\end{defn}
   
    \subsection{Variational and differential formulations}
	\label{sec:Section3}
    In the context of the kinematic Euler-Bernoulli assumptions for a metric graph, without any pre-stress or external force, the strain energy of the edge $e$ is expressed as
    \begin{equation}
    \label{eq:energyFunctional}
    %\begin{split}
    \cU^{(e)}(x) = \frac{1}{2} \int_e \big(
    a_e(x)|u'_e(x)|^2 + q_e(x)|u_e(x)|^2 \big) dx.
    %\end{split}
    \end{equation} 
    The integration here is over the edge $e$, parameterized by the arc length $x\in[0,\ell_e]$. Throughout the remainder of the manuscript, we assume that each beam in our frame is homogeneous in the axial direction, that is, $a_e(x) \equiv a$ is independent of $x$ and $e$. Extensions of all results to variable stiffness are straightforward. The corresponding energy at the vertex set due to the discontinuity of displacement and rotation fields is a functional of the form   
    \begin{equation}
    \label{eq:vertexEnergyFunctional}
    \cU^{(\vt)} = 
    \frac{1}{2} \sum_{e \sim \vt} \|u_\vt^{\hspace{0.5mm} \circ}- u_{e}(\vt) \|^2,
    \end{equation}
    and the total energy is expressed as
	\begin{equation}
	\label{eq:energyFunctionalGraph}
		\cU\up{G} := \sum_{e \in \cE} \cU^{(e)} + \sum_{\vt \in \cV} \cU^{(\vt)}.
	\end{equation}

    \begin{thm}
		\label{MainThm}
		Energy form~\eqref{eq:energyFunctionalGraph} with free semi-rigid vertices corresponds to the self-adjoint operator 
        $$\cA \colon L^2(G) \to L^2(G)$$ 
        with compact resolvent, acting as
		\begin{equation}
			\label{diffSystem}
			\Big( (u_e)_{e \in \cE} , (u_\vt^{\hspace{0.5mm}\circ})_{\vt \in \cV}\Big)
			~\mapsto~
			\Big( (-a u_e'' + q_e u_e)_{e \in \cE} ,(\mathfrak{m}_\vt^{-1} F_\vt )_{\vt \in \cV})\Big)
		\end{equation}
        on the set of edges $e \in \cE$ and vertices $\vt \in \cV$ of the graph $G$, where $F_{\vt}$ represents the net force at vertex $\vt$, i.e.
       \begin{equation}
		\label{eq:FM}
			 F_\vt = \displaystyle\sum_{e \sim \vt} s_e^\vt a u_e'.
		\end{equation}
         The domain of the operator $\cA$ consists of the functions $(u, \vec u_\vt^{\hspace{0.5mm}\circ})$ that belong to 
		\begin{equation}
		H_\cA(G) := 
		\prod_{e \in \cE} H^2(e)
		\times \prod_{\vt \in \cV} \bC^2(\vt)
		\end{equation}
		that satisfy at each vertex $\vt$ and for all $e \sim \vt$ non-homogeneous Robin conditions 
        \begin{equation}
            u_e(\vt) + s_e^\vt a \kappa^{-1} u_e'(\vt)  = u_\vt^{\hspace{0.5mm} \circ}
        \end{equation}
        defined in \eqref{eq:semiRigid_displacement} as the semi-rigidity condition.
	\end{thm}	

    \subsection{Interpretation of the self-adjoint operator}\label{subsecQuasimomentum}
     The eigenvalue equation of the self-adjoint operator in Theorem \ref{MainThm} reduces to
	\begin{equation}
	\label{eq:eigenProblemEta}
	\cA u_e(x) :=  -a u_e''(x) + q(x) u(x)= \lambda u_e(x)
	\end{equation} 
	subjected to the vertex conditions of 
        \begin{enumerate}
            \item \textit{semi-rigidity conditions}
	\begin{equation}
	\label{eq:FBVertexCondEta1}
		u_e(\vt) + s_e^{\vt} a \kappa^{-1} u_e'(\vt) = 
		u_{e'}(\vt) + s_{e'}^{\vt} a \kappa^{-1} u_{e'}'(\vt)
	\end{equation}
        for any pair $e,e' \sim \vt$, where $\kappa \in (0,\infty)$ and 
        \item \textit{net-force balance conditions}
	\begin{equation}
	\label{eq:FBVertexCondEta2}
		\displaystyle\sum_{e \sim \vt} s_e^{\vt} a u_e'(\vt)  = \lambda m \big(u_{e'}(\vt) + s_{e'}^{\vt} a \kappa^{-1} u_{e'}'(\vt)\big)
	\end{equation}
        for any $e' \sim \vt$, where $m \in (0,\infty)$.
        \end{enumerate}
	We note that in \eqref{eq:FBVertexCondEta2}, the right-hand side is the same for any $e' \sim \vt$, since the equality \eqref{eq:FBVertexCondEta1} holds. Moreover, one may notice that the spectrum is determined by the dispersion relation, and thus when the latter is known, the former can be determined as well.
    
    Floquet-Bloch theory reduces the study of Hamiltonian $\cA$ to the study of the Bloch Hamiltonian family $\cA^{\Theta}$ acting in $L^2(W)$ for the values of the \textit{quasimomentum} $\Theta = (\theta_1,\theta_2)$ in the Brillouin zone $\bT^2$. Here, the Bloch Hamiltonian $\cA^{\Theta}$ acts the same way $\cA$ does, but it is applied to a different space of functions.  Each function $u = \{u_e\}_{e \in E}$ in the domain of $\cA^{\Theta}$ must belong to the Sobolev space $u_e \in H^2(e)$ on each edge $e$ and satisfy the vertex conditions \eqref{eq:FBVertexCondEta1} and \eqref{eq:FBVertexCondEta2} as well as the cyclic conditions (Floquet-Bloch conditions)
	\begin{equation}
		\label{eq:FBThm}
		u(x + n_1\vec{b}_1+n_2\vec{b}_2)=  e^{i(n_1\theta_1+n_2\theta_2)}u(x) 
	\end{equation}
	for any $n_1, n_2 \in \bZ$ and any $x \in G$. Due to conditions \eqref{eq:FBThm}, function $u$ is uniquely determined by its restriction to the fundamental domain $W$. Then by a parametrization of $e_1$ as $[0,1]$ (and hence $\vt_1$ as $0$ and $\vt_2$ as $1$), conditions \eqref{eq:FBVertexCondEta1} and \eqref{eq:FBVertexCondEta2} at $\vt_1$ reduce to
	\begin{gather}
	u_1(0) - a \kappa^{-1} u_1'(0) = u_2(0) - a\kappa^{-1} u_2'(0) = u_3(0) - a\kappa^{-1} u_3'(0)  =: \omega_{0}, \label{leftcond1} \\
	-a(u_1'(0) + u_2'(0) + u_3'(0)) =  \lambda m  \omega_{0}. \label{leftcond2}
	\end{gather}
	Similarly, at vertex $\vt_2$ we get
	\begin{gather}
	u_1(1) + a\kappa^{-1} u_1'(1) = \big(u_2(1) + a\kappa^{-1} u_2'(1)\big) e^{i \theta_1} =  \big(u_3(1)+a\kappa^{-1} u_3'(1)\big) e^{i \theta_2}  =: \omega_{1}, \label{vertexv1cond1} \\
	a(u_1'(1) + u_2'(1)e^{i \theta_1}  + u_2'(1)e^{i \theta_2})  = \lambda m \omega_{1}. \label{vertexv1cond2}
	\end{gather}
	By standard arguments (see, e.g., \cite{K16}), $\cA^{\Theta}$ has a discrete spectrum $\sigma(\cA^{\Theta}) = {\lambda_j(\Theta)}$. The graph of the multi-valued function $\Theta \mapsto \{\lambda_j(\Theta)\}$ is known as the \textit{dispersion relation}, or \textit{Bloch variety} of the operator $\cA$. It is known that the range of this function is the spectrum of $\cA$:
	\begin{equation}
	\sigma(\cA) = \bigcup_{\Theta \in [-\pi,\pi]^2}\sigma(\cA^{\Theta}).
	\end{equation}
	Our goal now is to determine the spectrum of $\cA^{\Theta}$ and thus the dispersion relation of $\cA$. In order to determine this spectrum, we have to solve the eigenvalue problem
	\begin{equation}\label{2ndOrderOp}
	\cA^{\Theta}\Psi(x) = \lambda \Psi(x) 
	\end{equation}
	for $\lambda \in \mathbb{R}$ and non-trivial function $u_e(x) \in L_e^2(W)$ satisfying the above boundary conditions. Let us denote by $\Sigma_0$ the spectrum of operator $\cA$ on the interval $(0,1)$ with boundary conditions 
	%\begin{equation}
	%\label{eq:boundCondD}
	%u(0) - a\kappa^{-1}u'(0) = 0, \quad  u(1) + a\kappa^{-1}u'(1) = 0
	%\end{equation}
        \begin{align}
	u(0) - a\kappa^{-1}u'(0) &= 0, \label{eq:boundCondD0} \\
        u(1) + a\kappa^{-1}u'(1) &= 0. \label{eq:boundCondD1}
	\end{align}
	For $\lambda \notin \Sigma_0$, there exist two linearly independent solutions of the eigenvalue problem on $(0,1)$, denoted by $\psi_1(x)$ and $\psi_2(x)$ and depending on both $\lambda$ and $\kappa$ such that at boundaries they satisfy
	\begin{equation}
	\label{eq:etaIndSol}
	\begin{split}
	\psi_1(0) - a\kappa^{-1}\psi_1'(0) = 1, \qquad  \psi_1(1) + a\kappa^{-1}\psi_1'(1) = 0,\\
	\psi_2(0) - a\kappa^{-1}\psi_2'(0) = 0,\qquad \psi_2(1) + a\kappa^{-1}\psi_2'(1) = 1.
	\end{split}
	\end{equation}
	Linearly independent solutions satisfying \eqref{eq:etaIndSol} then admit representation of any solution $u$ of \eqref{eq:eigenProblemEta} from the domain of $\cA^\theta$ on each edge in $W$ and satisfying conditions \eqref{leftcond1} and \eqref{vertexv1cond1} as 
	\begin{align}
	&u_1(x) = \omega_{0} \psi_1(x) + \omega_{1} \psi_2(x), \label{u1cond} \\
	&u_2(x) = \omega_{0} \psi_1(x) + e^{-i\theta_1}\omega_{1} \psi_2(x), \label{u2cond} \\
	&u_3(x) = \omega_{0} \psi_1(x) + e^{-i\theta_2}\omega_{1} \psi_2(x). \label{u3cond}
	\end{align}
	Let us introduce the function
    \begin{equation}\label{Sfunction}
        S(\theta_1,\theta_2):= 1+e^{-i\theta_1}+e^{-i\theta_2}
    \end{equation}
    for $\theta_1, \theta_2 \in [-\pi,\pi]$.
    Then using \eqref{u1cond}, \eqref{u2cond} and \eqref{u3cond} in \eqref{leftcond2} and \eqref{vertexv1cond2} we get the system 
	\begin{equation}
	\label{eq:linearSystem}
	\begin{split}
	3\psi_1'(0) \omega_{0}+S(\theta_1,\theta_2) \psi_2'(0) \omega_{1} &= -\lambda m a^{-1} \omega_{0}, \\
	3\psi_2'(1) \omega_{1} + \overline{ S(\theta_1,\theta_2)} \psi_1'(1) \omega_{0} &= \lambda ma^{-1} \omega_{1}.
	\end{split}
	\end{equation}
	This relation can be further simplified by imposing the symmetry of the potential which implies 
	\begin{equation}\label{symemtryconditions}
	\psi'_1(0) = -\psi'_2(1), \qquad 	\psi'_1(1) = -\psi'_2(0),
	\end{equation}
	so \eqref{eq:linearSystem} reduces to finding the coefficient vector $(\omega_{0}, \omega_{1})^{\text{T}}$  in the kernel of the $2\times2$ matrix $\mathbb{M}(\lambda;\kappa^{-1},m)$ defined as
	\begin{equation}\label{coefmatrix}
	\begin{pmatrix}
		3\psi'_2(1) - \lambda m a^{-1} & - S(\theta_1,\theta_2)  \psi'_2(0) \\
		-\overline {S(\theta_1,\theta_2)} \psi'_2(0) &3\psi'_2(1) - \lambda m a^{-1}
	\end{pmatrix}.
	\end{equation}
	A non-trivial solution exists if and only if the matrix \eqref{coefmatrix} is singular, which is equivalent to a vanishing determinant as
	\begin{equation}
	\big(3\psi'_2(1) -\lambda m a^{-1}\big)^2 = (\psi'_2(0))^2 |S(\theta_1,\theta_2)|^2
	\end{equation}
	We summarize this result in the following proposition. 
	\begin{prop}
		\label{prop1}
		Let $\lambda \not \in \Sigma_0$, where $\Sigma_0$ is defined by \eqref{eq:boundCondD0} and \eqref{eq:boundCondD1}. Then $\lambda$ is in spectrum of the Hamiltonian $\cA$ if and only if there is $(\theta_1,\theta_2) \in [-\pi,\pi]^2$ such that 
		\begin{equation}
		\label{eq:PropResultDisp}
		\frac{\psi'_2(1) }{\psi'_2(0)} - \frac{\lambda m a^{-1}}{3\psi'_2(0)} = \pm \frac{|S(\theta_1,\theta_2)|}{3}.
		\end{equation}
	\end{prop}
	\begin{remark}
		Applying the property  that $\psi_2'(0) \not= 0$, condition \eqref{eq:PropResultDisp} for case $m = 0$ reduces to
		\begin{equation}
		\Big(\frac{\psi'_2(1) }{\psi'_2(0)} - \frac{|S(\theta_1,\theta_2)|}{3}\Big) \Big(\frac{\psi'_2(1) }{\psi'_2(0)} + \frac{|S(\theta_1,\theta_2)|}{3}\Big) = 0.
		\end{equation}
		This then has a similar form as the  Schr\"{o}dinger Hamiltonian on graphene reported in \cite{KP07} in which the dispersion relation contains Dirac points (conical singularities) representing one of the most interesting features of such structures. These singularities resemble Dirac spectra for massless fermions and thus lead to unusual physical properties of graphene. We will describe the Dirac points of $\cA$ in the next section.
	\end{remark}

     \begin{figure}[h]
		\centering
		\includegraphics[width=0.41\textwidth]{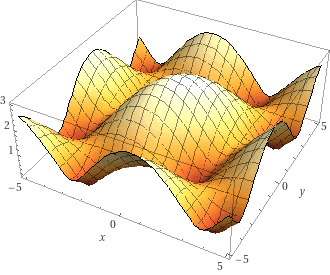} \qquad \qquad
		\includegraphics[width=0.33\textwidth]{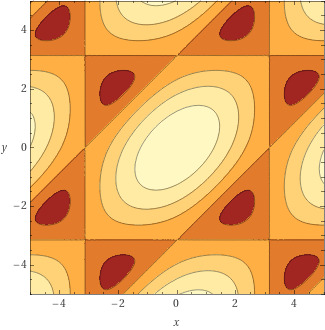}
        \caption{Plot and level curves of $|S(x,y)|$ defined in \eqref{Sfunction}.}
		\label{fig:SFunc}
	\end{figure}

        %\begin{figure}[h]
	%	\centering
	%	\includegraphics[width=0.35\textwidth]{sfcn1.jpg} \qquad \quad
	%	\includegraphics[width=0.27\textwidth]{sfcn2.jpg}
        %    \caption{Plot and level curves of $|S(\theta_1,\theta_2)|$.}
	%	\label{fig:|S|}
	%\end{figure}

        \section{Spectral analysis of the Hamiltonian $\cA$}
	\label{Section3}
 
	\subsection{\textbf{The dispersion relation of $\cA$}}
    
	Next, we want to interpret the functions $\psi'_2(1)/\psi'_2(0)$ and $1/\psi'_2(0)$ in terms of the original potential $q_0(x)$ on $[0, 1]$ using fundamental solutions. To this end, let us periodically extend $q_0$ to the whole real axis $\bR$ and consider the periodic operator $\cA^{\text{per}}$ in $\bR$ given as in \eqref{eq:eigenProblemEta} with the resulting periodic potential:
	\begin{equation}
	\label{eq:brokenHill}
	\cA^{\text{per}} u(x) = -a u''(x) + q_0(x)u(x).
	\end{equation}
	Without loss of generality, we assume $\cV = \bZ$. We introduce the \textit{jump operator} on $\vt \in \cV$ acting as 
	\begin{equation}
	\cM_{x_0}\big[g(x)\big] := \frac{1}{2}\Big(\lim_{x \rightarrow x_0^+}g(x) + \lim_{x \rightarrow x_0^-}g(x)\Big).
	\end{equation}
	Note that a possible discontinuity occurs only at the vertices, so if $x_0 \in G$ is not a vertex, then $$ \cM_{x_0}\big[g(x)\big] = g(x_0). $$
	%\begin{figure}[h]
	%	\centering
	%	\includegraphics[width=0.8\textwidth]{W3}
	%	\caption{The hexagonal lattice $G$ and a fundamental domain}
	%\end{figure}
    Next, we introduce the fundamental solutions $c_{\lambda}$ and $s_{\lambda}$ of \eqref{eq:brokenHill} satisfying 
	\begin{equation}
	\label{eq:jumpVertexConds}
	\begin{split}
	&\cM_{0}\big[c_{\lambda}(x)\big] = 1, \qquad \cM_{0}\big[s_{\lambda}(x)\big] = 0\\
	&\cM_{0}\big[c_{\lambda}'(x)\big] = 0, \qquad \cM_{0}\big[s_{\lambda}'(x)\big] = 1
	\end{split}
	\end{equation}
	The monodromy matrix has the form for $x \in \bR$
	\begin{equation}
	\label{eq:monodMatrix}
	    \begin{pmatrix}
	    \cM_{1}\big[c_{\lambda}(x)\big] & \cM_{1}\big[s_{\lambda}(x)\big]\\
            \cM_{1}\big[c_{\lambda}'(x)\big] & \cM_{1}\big[s_{\lambda}'(x)\big]
	\end{pmatrix}
	\end{equation}
	and it shifts by the period along the solutions of \eqref{eq:brokenHill}. The monodromy matrix $M(\lambda)$ is an entire function of $\lambda \in \mathbb{C}$ and maps $\mathbb{R}$ to $\mathbb{R}$.	
 
        Our goal is to obtain solutions $\psi_1$ and $\psi_2$ that we defined by conditions \eqref{eq:etaIndSol} in terms of $c_{\lambda}$ and $s_{\lambda}$, and hence get the dispersion relation for $\cA$ in terms of the fundamental solution using Proposition \ref{prop1}. 
        
	Now for $x_0 \in [0,1]$, writing 
	\begin{align}
	\cM_{x_0}\big[c_{\lambda}(x)\big] &= \alpha_c \lim_{x \rightarrow x_0} \psi_1(x) + \beta_c \lim_{x \rightarrow x_0} \psi_2(x) \label{limit1} \\
        \cM_{x_0}\big[s_{\lambda}(x)\big] &= \alpha_s \lim_{x \rightarrow x_0} \psi_1(x) + \beta_s \lim_{x \rightarrow x_0} \psi_2(x) \label{limit2}
	\end{align}
	and applying the properties \eqref{eq:jumpVertexConds} to find the coefficients of \eqref{limit1}, we get
    \begin{align}
        1 &= \alpha_c \psi_1(0) + \beta_c \psi_2(0)\\
        0 &= \alpha_c\psi_1'(0) + \beta_c \psi_2'(0)
    \end{align}
    and then using conditions \eqref{eq:boundCondD0} and \eqref{eq:boundCondD1}, we have
    \begin{align}
        1 &= \alpha_c [1+a\kappa^{-1}\psi_1'(0)] + \beta_c a\kappa^{-1}\psi_2'(0)\\
        0 &= \alpha_c\psi_1'(0) + \beta_c \psi_2'(0)
    \end{align}
    and finally solving this system and recalling the symmetry properties \eqref{symemtryconditions} we obtain
    \begin{equation}\label{Mcoef1}
       \cM_x[c_{\lambda}(x)] = \psi_1(x) + \frac{\psi_2'(1)}{\psi_2'(0)}\psi_2(x).
    \end{equation}
    Similarly, applying the properties \eqref{eq:jumpVertexConds} to find the coefficients of \eqref{limit2}, we get 
     \begin{align}
        0 &= \alpha_s \psi_1(0) + \beta_s \psi_2(0)\\
        1 &= \alpha_s\psi_1'(0) + \beta_s \psi_2'(0)
    \end{align}
    and then using conditions \eqref{eq:boundCondD0} and \eqref{eq:boundCondD1}, we have
    \begin{align}
        0 &= \alpha_c [1+a\kappa^{-1}\psi_1'(0)] + \beta_c a\kappa^{-1}\psi_2'(0)\\
        1 &= \alpha_c\psi_1'(0) + \beta_c \psi_2'(0)
    \end{align}
    and finally solving this system and recalling the symmetry properties \eqref{symemtryconditions} we obtain
    \begin{equation}\label{Mcoef2}
       \cM_x[s_{\lambda}(x)] = -a\kappa^{-1}\psi_1(x) + \frac{1-a\kappa^{-1}\psi_2'(1)}{\psi_2'(0)}\psi_2(x).
    \end{equation}

	  Note that at the endpoints, the limits in \eqref{limit1} and \eqref{limit2} are one-sided. Now, using \eqref{Mcoef1}, we observe that 
	\begin{equation}\label{M1g1coef}
	\cM_{1}\big[c_{\lambda}(x)\big]+a\kappa^{-1} \cM_{1}\big[c_{\lambda}'(x)\big] = \big(\underbrace{\psi_1(1)+a\kappa^{-1} \psi_1'(1)}_{=~0}\big) + \frac{\psi_2'(1)}{\psi_2'(0)} \big(\underbrace{\psi_2(1)+a\kappa^{-1} \psi_2'(1)}_{=~1}\big) = \frac{\psi_2'(1)}{\psi_2'(0)}
	\end{equation}
	and similarly, 
	\begin{align}\label{M1g2coef}
	&\cM_{1}\big[s_{\lambda}(x)\big]+a\kappa^{-1} \cM_{1}\big[s_{\lambda}'(x)\big]\\
    &=-a\kappa^{-1}\big(\underbrace{\psi_1(1)-a\kappa^{-1} \psi_1'(1)}_{=~0}\big) + \frac{[1-a\kappa^{-1}\psi_2'(1)]\psi_2'(1)}{\psi_2'(0)} \big(\underbrace{\psi_2(1)+a\kappa^{-1} \psi_2'(1)}_{=~1}\big)\\
    &= \frac{1-a\kappa^{-1}\psi_2'(1)}{\psi_2'(0)}.
	\end{align}
	We remark that any solution $g(x)$ to the eigenvalue problem $\cA g = \lambda g$ can be written as a linear combination of $c_{\lambda}(x)$ and $s_{\lambda}(x)$ since the Wronskian at a vertex, let us say $x_n = 0$, satisfies
	\begin{equation}
	\cW_{0}(c_{\lambda},s_{\lambda}) = 
	\cM_{0}\big[c_{\lambda}(x)\big]\cM_{0}\big[s_{\lambda}'(x)\big] - \cM_{0}\big[s_{\lambda}(x)\big]\cM_{0}\big[c_{\lambda}'(x)\big] = 1
	\end{equation}
	and is non-vanishing. 
    
    Next, we formalize the result in Proposition \ref{prop2}. Before that, let us introduce the notation
	\begin{equation}
    \label{GMatrix}
	\mathbb{M} := 
	\begin{pmatrix}
	\cM_{1}\big[c_{\lambda}(x)\big]  & \cM_{1}\big[s_{\lambda}(x)\big] \\
	\cM_{1}\big[c_{\lambda}'(x)\big] & \cM_{1}\big[s_{\lambda}'(x)\big]
	\end{pmatrix}
	\end{equation}
	for the monodromy matrix and the unit vectors
    \begin{equation}
        \vv{e}_1 = \begin{pmatrix}
            1\\
            0
        \end{pmatrix} \qquad \text{and} \qquad \vv{e}_2 = \begin{pmatrix}
            0\\
            1
        \end{pmatrix}.
    \end{equation}
	\begin{prop}
		\label{prop2}
		If $\lambda \not \in \Sigma_0$, then $\lambda$ is in spectrum of the Hamiltonian $\cA$ if and only if there is $(\theta_1,\theta_2) \in [-\pi,\pi]^2$ such that 
		\begin{equation}
		\label{eq:PropResultDisp2}
		\Big[(1~a\kappa^{-1}) \mathbb{M}\vv{e}_1\Big] - \frac{\lambda m}{3a} \Big[(1~a\kappa^{-1}) \mathbb{M}\begin{pmatrix}
		    a\kappa^{-1}\\
            1
		\end{pmatrix}\Big] = \pm \frac{|S(\theta_1,\theta_2)|}{3},
		\end{equation}
        where $\mathbb{M}$, $\Sigma_0$ and $S(\theta_1,\theta_2)$ are defined in \eqref{GMatrix}, \eqref{eq:boundCondD0} and \eqref{eq:boundCondD1}, and \eqref{Sfunction} respectively.
	\end{prop}

        \begin{proof}[\normalfont \textbf{Proof of Proposition~\ref{prop2}}]
            From equations \eqref{M1g1coef} and \eqref{M1g2coef}, the left-hand side of \ref{eq:PropResultDisp2} becomes
            \begin{align}
                &\Big[(1~a\kappa^{-1}) \mathbb{M}\vv{e}_1\Big] - \frac{\lambda m}{3a} \Big[(1~a\kappa^{-1}) \mathbb{M}\begin{pmatrix}
		    a\kappa^{-1}\\
            1
		\end{pmatrix}\Big]\\ &= 
                \cM_{1}\big[c_{\lambda}(x)\big]+a\kappa^{-1} \cM_{1}\big[c_{\lambda}'(x)\big]\\ 
                & -\frac{\lambda m}{3a} \Big( a\kappa^{-1}\big(\cM_{1}\big[c_{\lambda}(x)\big]+a\kappa^{-1} \cM_{1}\big[c_{\lambda}'(x)\big]\big) + \cM_{1}\big[s_{\lambda}(x)\big]+a\kappa^{-1} \cM_{1}\big[s_{\lambda}'(x)\big] \Big)\\
                &= \frac{\psi_2'(1)}{\psi_2'(0)} -  \frac{\lambda m}{3a} \frac{1}{\psi_2'(0)} \label{prooffinalline}.
            \end{align}
            However, \eqref{prooffinalline} is nothing more than the left-hand side of \eqref{eq:PropResultDisp}. Hence, we get the result we look for.
        \end{proof}
 
	\begin{remark}
		For the case $m \rightarrow 0$ (no concentrated mass at the vertices) and $\kappa^{-1} \rightarrow 0$ (rigid vertices), the result of Proposition \ref{prop2} reduces to the one reported in \cite{KP07}. In fact in this case  
		\begin{equation}
		\mathbb{M} := 
		\begin{pmatrix}
		c_{\lambda}(1) & s_{\lambda}(1) \\
		c_{\lambda}'(1) & s_{\lambda}'(1)
		\end{pmatrix}
		\end{equation}
		and applying the fact that $c_{\lambda}(1) = s_{\lambda}'(1) = \psi_2'(1)/\psi_2'(0)$, relation \eqref{eq:PropResultDisp2} can be replaced by 
		\begin{equation}
		\label{eq:PropResultDisp3}
		\bigg(\frac{1}{2}\tr(\mathbb{M}) - \frac{|S(\theta_1,\theta_2)|}{3}\bigg)
		\bigg(\frac{1}{2}\tr(\mathbb{M}) + \frac{|S(\theta_1,\theta_2)|}{3}\bigg)
		= 0.
		\end{equation}
	\end{remark}
	Next, referring to \eqref{eq:PropResultDisp}, we will fully characterize the spectrum of the Hamiltonian $\cA$ \eqref{eq:eigenProblemEta} with general vertex models \eqref{eq:FBVertexCondEta1}, \eqref{eq:FBVertexCondEta2}.
	
	\subsection{Spectral properties of $\cA$}
    
	Notice that the solution to \eqref{eq:PropResultDisp} implicitly depends on both $\lambda$ and $\kappa$ since $\psi_1(x)$ and $\psi_2(x)$ do. Proposition \ref{prop1}, in particular, says that in order to find the spectrum of $\cA$, we need to calculate the range of $S(\theta_1,\theta_2)$ on $[-\pi,\pi]^2$. 

    The function $S(\theta_1,\theta_2)$ is the same function that appeared in the study of the graphene Hamiltonian \cite{KP07} and other quantum graphs on hexagonal lattices \cite{EH22}. The range of $|S(\theta_1,\theta_2)|$ is $[0,3]$, it has minima at $\pm (2\pi/3,-2\pi/3)$ and maximum at $(0,0)$. These properties can be observed through the identity
	\begin{equation}
		|S(\theta_1,\theta_2)|^2 = |1+e^{i \theta_1} + e^{i\theta_2}|^2 = 1 + 8 \cos\Big(\frac{\theta_1-\theta_2}{2}\Big) \cos\Big(\frac{\theta_1}{2}\Big) \cos\Big(\frac{\theta_2}{2}\Big).
	\end{equation}
	See Figure \ref{fig:SFunc} for a plot and the level curves of this function.
        
        Note that we obtained the dispersion relation in Proposition \ref{prop2} for $\lambda \notin \Sigma_0$, so we study the eigenvalues $\lambda \in \Sigma_0$ separately in Lemma \ref{sigmaDLem}. Before starting our discussion on $\Sigma_0$, we introduce some notation:
        \begin{align}\label{D012Definition}
            \cD_1(\lambda) &:= (1~a\kappa^{-1}) \mathbb{M}\vv{e}_1 = \cM_{1}\big[c_{\lambda}(x)\big]+a\kappa^{-1} \cM_{1}\big[c_{\lambda}'(x)\big],\\
            \cD_2(\lambda) &:= (1~a\kappa^{-1}) \mathbb{M}\vv{e}_2 = \cM_{1}\big[s_{\lambda}(x)\big]+a\kappa^{-1} \cM_{1}\big[s_{\lambda}'(x)\big],\\
            \cD_0(\lambda) &:= (1~a\kappa^{-1}) \mathbb{M}\begin{pmatrix}
		    a\kappa^{-1}\\
            1
		\end{pmatrix} = \cD_2(\lambda) + a\kappa^{-1}\cD_1(\lambda),
        \end{align}
        where $c_{\lambda}$ and $s_{\lambda}$ are the fundamental solutions we introduced in the previous section.
        
        Recall that $\Sigma_0$ is the spectrum of $\cA$ on the interval $(0,1)$ with the boundary conditions \eqref{eq:boundCondD0} and \eqref{eq:boundCondD1}. Assuming $u$ to be such a solution and representing it in terms of the fundamental solutions as
        \begin{equation}
            u(x) = \alpha c_{\lambda}(x) + \beta s_{\lambda}(x),
        \end{equation}
        condition \eqref{eq:boundCondD0} implies that
        \begin{equation}
            \alpha \Big(\cM_{0}\big[c_{\lambda}(x)\big] - a\kappa^{-1} \cM_{0}\big[c_{\lambda}'(x)\big]\Big) + \beta \Big(\cM_{0}\big[s_{\lambda}(x)\big] - a\kappa^{-1} \cM_{0}\big[s_{\lambda}'(x)\big]\Big) = 0,
        \end{equation}
        which is
        \begin{equation}
            \alpha = \beta a\kappa^{-1}.
        \end{equation}
        Therefore, \eqref{eq:boundCondD1} that has the form
        \begin{equation}
            \alpha \cD_1(\lambda) + \beta \cD_2(\lambda) = 0,
        \end{equation}
        becomes 
        \begin{equation}
            \beta [a\kappa^{-1}\cD_1(\lambda) + \cD_2(\lambda)] = 0.
        \end{equation}
        Therefore, we get a representation of $\Sigma_0$ in terms of the fundamental solutions as
        \begin{equation}\label{Sigma0Definition}
            \Sigma_0 = \{ \lambda \in \mathbb{R} ~|~ \cD_0(\lambda) = 0 \}.
        \end{equation}
        Next, let us show that $\Sigma_0$ is a subset of the pure-point spectrum of $\cA$.
        \begin{lem}
		\label{sigmaDLem}
		Each point $\lambda \in \Sigma_0$ is an eigenvalue of infinite multiplicity of the Hamiltonian $\cA$. The corresponding eigenspace is generated by simple loop states, i.e. by eigenfunctions which are supported on a single hexagon and vanish at the vertices.
	\end{lem}
	\begin{proof}[\normalfont \textbf{Proof of Lemma~\ref{sigmaDLem}}] 
		Let us first show that each $\lambda \in \Sigma_0$ is an eigenvalue. Let $u(x)$ be an eigenfunction of the operator $\cA$ with the eigenvalue $\lambda \in \Sigma_0$ satisfying the boundary conditions \eqref{eq:boundCondD0} and \eqref{eq:boundCondD1} on $[0,1]$. Note that $u(1-x)$ is also an eigenfunction with the same eigenvalue. This is obvious for the operator since $u''(1-x) = u''(x)$ and moreover, the potential is symmetric. Moreover, for the boundary conditions, this becomes clear if one changes the direction of edge on $[0,1]$ domain. So, we can assume that $u(x)$ is odd (or even). If $u(x)$ is neither even nor odd, then $u(x) - u(1-x)$ is an odd eigenfunction. For an odd eigenfunction, repeating it on each of the six edges of a hexagon and letting the eigenfunction be zero on any other hexagon, we get an eigenfunction of the operator $\cA$.  Observe that by this construction the first condition in \eqref{eq:FBVertexCondEta1} is met. Regarding the second term in \eqref{eq:FBVertexCondEta2}, we stress that its right-hand side is zero due to the fact that $\lambda \in \Sigma_0$, which will be compatible with its left-hand side sum due to the construction continuity of derivatives stem from odd property of eigenfunction on $[0,1]$. If $u$ is an even eigenfunction, then repeating it around the hexagon with an alternating sign and letting the eigenfunction to be zero on any other hexagon, we get an eigenfunction of the operator $\cA$. Therefore, $\lambda \in \Sigma_0$. We get the rest of the proof following the arguments of Lemma 3.5 in \cite{KP07}.
	\end{proof}

        Observe that for $\lambda \not \in \Sigma_0$, the solution of the eigenvalue problem on $(0,1)$ satisfying boundary conditions \eqref{eq:etaIndSol} can be explicitly determined. In fact, for $k = 1,2$ the solutions are represented by $$\psi_k(x) = A_k s_{\lambda}(x) + B_k c_{\lambda}(x),$$ in which coefficients $A_1$ and $A_2$ satisfy 
	\begin{align}
	A_1 &= -\frac{\cD_1(\lambda)}{\cD_0(\lambda)}, \label{eq:A1} \\
        A_2 &= \frac{1}{\cD_0(\lambda)}, \label{eq:A2}
	\end{align}
        while 
        \begin{align}
	B_1 &= 1+a\kappa^{-1} A_1, \label{eq:B1} \\
        B_2 &= a\kappa^{-1} A_2. \label{eq:B2}
	\end{align}
        
        We remark that due to restriction $\lambda \not \in \Sigma_0$, the denominators in the coefficients are never zero, including $\lambda = 0$. 
	Now observe that solution for $\psi_2(x)$ with \eqref{eq:A2} and \eqref{eq:B2} implies
	\begin{equation}
	\begin{split}
	\psi_2'(0) &=  \frac{1}{\cD_0(\lambda)},\\
	\psi_2'(1) &= \frac{
	\cD_1(\lambda)}{\cD_0(\lambda)}.
	\end{split}
	\end{equation}
	Putting these characteristics in conjuction with multiple simplifications, the relation \eqref{eq:PropResultDisp} reduces to 
	\begin{equation}
	\label{eq:DR}
	\Delta^2(\lambda;\kappa,m,a) = |S(\theta_1,\theta_2)|^2/9
	\end{equation}
	with a (discriminant) type function of the form 
	\begin{equation}\label{eq:DeltaFunctionT1T2}
	   \Delta(\lambda;\kappa,m,a) := T_1(\lambda;\kappa,a) - \frac{m}{3} T_2(\lambda;\kappa,a),  
	\end{equation}
        where $T_1$ and $T_2$ are defined as 
	\begin{equation}
	\label{eq:T1T2}
	\begin{split}
	T_1(\lambda;\kappa,a) &:= \cD_1(\lambda), \\
	T_2(\lambda;\kappa,a) &:= \frac{\lambda}{a} \cD_0(\lambda).
	\end{split}
	\end{equation}

        Next, we describe the pure point, absolutely continuous and singular continuous spectra, and the dispersion relation of $\cA$. 
	\begin{thm}
		\label{thm:DispThm}
		Let  $\cA$ be the Hamiltonian defined in \eqref{eq:eigenProblemEta} satisfying the vertex conditions \eqref{eq:FBVertexCondEta1} and \eqref{eq:FBVertexCondEta2}. Then we have the following porperties of $\cA$.
		\begin{itemize}
			\item[(i)] The singular continuous spectrum $\sigma_{sc}(\cA)$ is empty. 
			\item[(ii)] The absolutely continuous spectrum $\sigma_{ac}(\cA)$ has band-gap structure and as a set is represented by
			\begin{equation}
			\sigma_{ac}(\cA) = \big\{\lambda \in \bR ~:~ \Delta(\lambda;\kappa,m,a) \in [-1,1]
			\big\}
			\end{equation}
			where $\Delta(\lambda;\kappa,m,a) = T_1(\lambda;\kappa,a) -\frac{m}{3} T_2(\lambda;\kappa,a)$ is defined in \eqref{eq:DeltaFunctionT1T2} and \eqref{eq:T1T2}.
			\item[(iii)] The pure point spectrum $\sigma_{pp}(\cA)$ coincides with the set $\Sigma_0$ that is defined in \eqref{Sigma0Definition}, and the eigenvalues are of infinite multiplicity with the corresponding eigenspaces generated by simple loop eigenstates. 
			\item[(iv)] The dispersion relation of $\cA$ consists of the variety 
			\begin{equation}
			\label{eq:PropResult}
			\Delta^2(\lambda;\kappa,m,a) = |S(\theta_1,\theta_2)|^2/9
			\end{equation}
			and the collection of flat branches for $\lambda \in \Sigma_0$.
		\end{itemize}
	\end{thm}
	\begin{proof}[\normalfont \textbf{Proof of Theorem~\ref{thm:DispThm}}] 
		Proofs of the items above are based on the tools developed in this paper along with already-established results in our references. For item $(i)$, observe that the singular continuous spectrum is empty, since $\cA$ is a self-adjoint elliptic operator (see e.g. Corollary 6.11 in \cite{K16}). The proof of $(ii)$ is based on Proposition \ref{prop1} and results in \eqref{eq:DR} and Lemma \ref{sigmaDLem}. Observe that since the range of $|S(\theta_1,\theta_2)|$ is $[0,3]$, then \eqref{eq:DR} is valid if and only if $\Delta(\lambda;\kappa,m) \in [-1,1]$. According to the Thomas’ analytic continuation argument, the eigenvalues correspond to the constant branches of the dispersion relation \cite{KP07,RS78,T73}. Since constant branches only occur at $\Sigma_0$, we get $\sigma_{\text{pp}}(\cA) \subseteq \Sigma_0$. But Lemma \ref{sigmaDLem} shows that $\Sigma_0 \subseteq \sigma_{\text{pp}}(\cA)$, so item $(iii)$ is the result of Lemma \ref{sigmaDLem}. 
		Finally, item $(iv)$ is the result of Proposition \ref{prop1}, identity \eqref{eq:DR}, and Lemma \ref{sigmaDLem}.
	\end{proof}
	The next proposition proves the existence of Dirac points of $\cA$.
	\begin{prop}
		\label{cor:DiracPoints}
		The set of Dirac points of $\cA$ is given by 
		\begin{equation}
		\big\{(\theta_1,\theta_2,\lambda) \in \bR^3 ~|~ (\theta_1,\theta_2) = \pm (2\pi/3,-2\pi/3) \quad \text{ and }\quad \Delta(\lambda;\kappa,m,a) = 0 \big\}.
		\end{equation}
		In other words, the dispersion surface of $\cA$ has conical singularities at any $\lambda$ satisfying
		\begin{equation}\label{CorDispD1D2}
		\cD_1(\lambda) - \frac{m}{3}\frac{\lambda}{a}\cD_0(\lambda) = 0,
		\end{equation}
        where $\cD_0$ and $\cD_1$ are defined in \eqref{D012Definition}.
	\end{prop}
	\begin{proof}[\normalfont \textbf{Proof of Proposition~\ref{cor:DiracPoints}}] 
        Recall that the dispersion relation is
        \begin{equation}\label{CorproofDisp}
		\Delta(\lambda;\kappa,m,a) = \pm |S(\theta_1,\theta_2)|/3.
	\end{equation}
        The right-hand side function $|S(\theta_1,\theta_2)|$ has non-degenerate minima $S(\theta_1,\theta_2) = 0$ (only) at the points $\pm(2\pi/3,-2\pi/3)$. Thus, the function $\pm |S(\theta_1,\theta_2)|$ has conical singularities at these points; see Figure \ref{fig:SFunc}. Since the left-hand side function $\Delta(\lambda;\kappa,m,a)$ depends on the entries of the monodromy matrix $\mathbb{M}$, which is an entire function of $\lambda$, we get the result we look for. Note that \eqref{CorDispD1D2} is nothing more than $$\Delta(\lambda;\kappa,m,a) = 0.$$
	\end{proof}

    \begin{remark}
        The identity \eqref{CorDispD1D2} takes the following forms in the corresponding special cases:
            \begin{enumerate}
                \item \textit{Rigid vertices with no concentrated mass ($\kappa^{-1} = 0$, $m = 0$):}
                \begin{equation}\label{GrapheneDirac}
                    c_{\lambda}(1) = 0
                \end{equation}
                \item \textit{Rigid vertices with concentrated mass ($\kappa^{-1} = 0$, $m > 0$):}
                \begin{equation}
                    c_{\lambda}(1) - \frac{m}{3}\frac{\lambda}{a}s_{\lambda}(1) = 0
                \end{equation}
                \item \textit{Semi-rigid vertices with no concentrated mass ($\kappa^{-1} > 0$, $m = 0$):}
                \begin{equation}
                    \cM_{1}\big[c_{\lambda}(x)\big] + a\kappa^{-1}\cM_{1}\big[c_{\lambda}'(x)\big] = 0
                \end{equation}
            \end{enumerate}
        Note that \eqref{GrapheneDirac} agrees with the Dirac points of graphene Hamiltonian shown in Corollary 3.7 of \cite{KP07}.
    \end{remark}

    \section{Spectral analysis of the free operator $\cA_0$}
	\label{Section4}

In this section, we consider the spectral properties we discussed in the previous sections for the free operator $\cA_0$. The presentation of the spectral analysis for this special case is in favor of readers with more applied background. Moreover, it will allow us to highlight the following observations:
\begin{enumerate}
    \item Classical Borg's theorem says that the one dimensional periodic Schr\"{o}dinger operator has no non-degenerate spectral gaps if and only if the potential function $q$ is constant \cite{B46}. The same result is valid for the graphene Hamiltonian \cite{KP07}. The study of Borg-type theorems in various continuous and discrete settings is an active research topic, see \cite{L23} and references therein. 
    
    In this section, we will observe that $\cA_0$ has non-degenerate spectral gaps, which means that the existence of semi-rigidity ($\kappa^{-1} > 0$) or concentrated mass ($m > 0$) creates spectral non-degenerate gaps even for the zero (and hence constant) potential.
    \item Recall from the previous section that $\Sigma_0$ characterizes the eigenvalues $\cA$. In this section, we will see that $\Sigma_0$ for the zero potential belongs to the set of end points of the spectral bands, that is 
    \begin{equation}
        \Sigma_0 \subseteq \{ \lambda \in \mathbb{R} ~|~ \Delta(\lambda;\kappa,m,a) = -1 \quad \text{or} \quad \Delta(\lambda;\kappa,m,a) = 1\}.
    \end{equation}
    \item We will also describe the dispersion relation, the spectrum, eigenvalues and the Dirac points in terms of $\cos(\sqrt{\lambda/a})$ and $\sin(\sqrt{\lambda/a})/\sqrt{\lambda/a}$.
\end{enumerate}

Let us start by providing the functions and notation that appeared in our result in previous sections for $\cA_0$. First we introduce $\mu := \sqrt{\lambda/a}$. Then, since the fundamental solutions for the zero potential are $c_{\lambda}(x) = \cos(\mu x)$ and $s_{\lambda}(x) = \sin(\mu x)/\mu$, the monodromy matrix becomes
\begin{equation}
    \mathbb{M} := \begin{pmatrix}
        \cos(\mu) & (1/\mu)\sin(\mu)\\
        -\mu\sin(\mu) & \cos(\mu)
    \end{pmatrix}
\end{equation}
and hence $\det(\mathbb{M}) = 1$ and $\tr(\mathbb{M}) = 2\cos(\mu)$. Therefore, the functions $\cD_k$ we introduced in \eqref{D012Definition} become
\begin{align}
    \cD_1(\lambda) &= \cos(\mu) - a\kappa^{-1}\mu^2\frac{\sin(\mu)}{\mu},\label{D1free}  \\
    \cD_2(\lambda) &= a\kappa^{-1}\cos(\mu) + \frac{\sin(\mu)}{\mu}, \label{D2free} \\
    \cD_0(\lambda) &= 2a\kappa^{-1}\cos(\mu) + \big[1 - (a\kappa^{-1})^2\mu^2\big]\frac{\sin(\mu)}{\mu}. \label{D0free}
\end{align}
Let us also recall that 
\begin{equation}
    \Sigma_0 = \{\lambda \in \mathbb{R} ~|~ \cD_0(\lambda) = 0\}
\end{equation}
and 
\begin{equation}\label{Deltafree}
    \Delta(\lambda;\kappa,m,a) = \cD_1(\lambda) - \frac{m}{3}\mu^2\cD_0(\lambda).
\end{equation}

Now, we can state our spectral results from the previous section for $\cA_0$.

\begin{lem}
		\label{sigmaDLemfree}
		Each point $\lambda \in \Sigma_0$, that is satisfying
        \begin{equation}
            \big[1 - (a\kappa^{-1})^2(\lambda/a)\big]\sin\big(\sqrt{\lambda/a}\big) + 2a\kappa^{-1}\cos\big(\sqrt{\lambda/a}\big) = 0
        \end{equation}
        is an eigenvalue of infinite multiplicity of the Hamiltonian $\cA_0$. The corresponding eigenspace is generated by simple loop states, i.e. by eigenfunctions which are supported on a single hexagon and vanish at the vertices.
	\end{lem}

\begin{thm}
		\label{thm:DispThmfree}
		Let  $\cA_0$ be the Hamiltonian defined in \eqref{eq:eigenProblemEta} with $q \equiv 0$ satisfying the vertex conditions \eqref{eq:FBVertexCondEta1} and \eqref{eq:FBVertexCondEta2}. Then we have the following properties of $\cA_0$.
		\begin{itemize}
			\item[(i)] The singular continuous spectrum $\sigma_{sc}(\cA_0)$ is empty. 
			\item[(ii)] The absolutely continuous spectrum $\sigma_{ac}(\cA_0)$ has band-gap structure and as a set is represented by $\lambda \in \mathbb{R}$ satisfying
			\begin{equation}
			-1 \leq \bigg[1 + 2(a\kappa^{-1})\frac{m}{3}\frac{\lambda}{a}\bigg]\cos(\sqrt{\lambda/a}) - \bigg[a\kappa^{-1}(\lambda/a) + \big[1 - (a\kappa^{-1})^2(\lambda/a)\big]\frac{m}{3}\frac{\lambda}{a} \bigg]\frac{\sin(\sqrt{\lambda/a})}{\sqrt{\lambda/a}} \leq 1.
			\end{equation}
			\item[(iii)] The pure point spectrum $\sigma_{pp}(\cA_0)$ coincides with the set $$\Sigma_0 = \Big\{ \lambda \in \mathbb{R} ~\Big|~ \big[1 - (a\kappa^{-1})^2(\lambda/a)\big]\frac{\sin\big(\sqrt{\lambda/a}\big)}{\sqrt{\lambda/a}} + 2a\kappa^{-1}\cos\big(\sqrt{\lambda/a}\big) = 0\Big\},$$ 
            and the eigenvalues are of infinite multiplicity with the corresponding eigenspaces generated by simple loop eigenstates. 
			\item[(iv)] The dispersion relation of $\cA_0$ consists of the variety 
			\begin{equation}
			\label{eq:PropResultfree}
			\bigg[1 + 2(a\kappa^{-1})\frac{m}{3}\frac{\lambda}{a}\bigg]\cos(\sqrt{\lambda/a}) - \bigg[a\kappa^{-1}(\lambda/a) + \big[1 - (a\kappa^{-1})^2(\lambda/a)\big]\frac{m}{3}\frac{\lambda}{a} \bigg]\frac{\sin(\sqrt{\lambda/a})}{\sqrt{\lambda/a}} = \pm\frac{|S(\theta_1,\theta_2)|}{3}
			\end{equation}
			and the collection of flat branches for $\lambda \in \mathbb{R}$ satisfying
            \begin{equation}
                \big[1 - (a\kappa^{-1})^2(\lambda/a)\big]\frac{\sin\big(\sqrt{\lambda/a}\big)}{\sqrt{\lambda/a}} + 2a\kappa^{-1}\cos\big(\sqrt{\lambda/a}\big) = 0.
            \end{equation}
		\end{itemize}
	\end{thm}

\begin{remark}\label{DeltaFreeRemark}
		For the case $m \rightarrow 0$ (no concentrated mass at the vertices) and $\kappa^{-1} \rightarrow 0$ (rigid vertices), for the zero potential, we have $\tr(\mathbb{M}) = \cos\big(\sqrt{\lambda/a}\big)$ and the dispersion relation becomes 
		\begin{equation}
		\label{eq:DG}
		\cos\big(\sqrt{\lambda/a}\big) = \pm \frac{|S(\theta_1,\theta_2)|}{3}.
		\end{equation}
		This is the result corresponding to the graphene Hamiltonian with zero potential; see Proposition 3.4 and Theorem 3.6 in \cite{KP07} for a detailed discussion.
\end{remark}

\begin{remark}\label{BorgFreeRemark}
    Figure \ref{fig:DeltaFunc} plots functions $\Delta(\lambda;\kappa,m,a)$ for selected parameters. As we can see in the first graph and \eqref{eq:DG}, the free graphene Hamiltonian has no non-degenerate spectral gaps, in other words all spectral bands are touching (but not overlapping). This property is characterized by having a constant potential for the one dimensional periodic Schr\"{o}dinger operator and is called the classical Borg's theorem. For $\cA_0$ we observe that even for the zero potential there exist non-degenerate spectral gaps if one of the parameters $\kappa^{-1}$ or $m$ is non-zero.  
\end{remark}

\begin{figure}[h]
            \centering
		\includegraphics[width=1.05\textwidth]{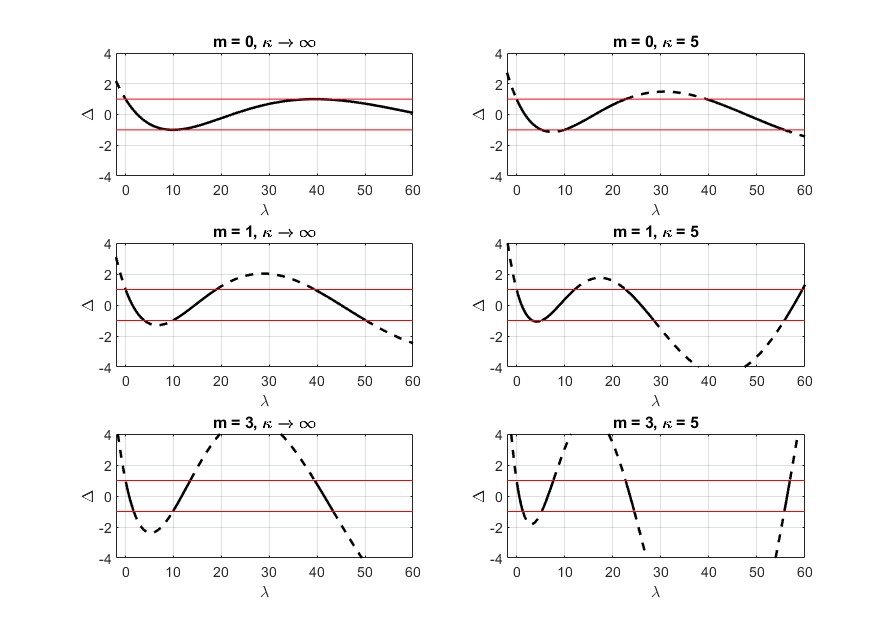}
		\caption{Plots of $\Delta$ as a function of $\lambda$ for $a=1$, composed of functions $T_1$ and $T_2$ as defined in \eqref{eq:DeltaFunctionT1T2}. Preimages of the solid curves give the spectral bands.}
		\label{fig:DeltaFunc}
	\end{figure}

Before  we represent the Dirac points and the endpoints of the spectral bands, we need a technical lemma.

\begin{lem}\label{D1D0identity}
For the free operator $\cA_0$ we have the identity
\begin{equation}
   [\cD_1(\lambda)]^2 = 1 - \sqrt{\lambda/a}\sin\Big(\sqrt{\lambda/a}\Big)\cD_0(\lambda).
\end{equation}
\end{lem}

\begin{proof}
    Using \eqref{D1free} and \eqref{D0free} we get
    \begin{align}
        [\cD_1(\lambda)]^2 &= [\cos(\mu) - a\kappa^{-1}\mu\sin(\mu)]^2\\
        &= \cos^2(\mu) - 2a\kappa^{-1}\mu\sin(\mu)\cos(\mu) + (a\kappa^{-1})^2\mu^2\sin^2(\mu)\\
        &= 1 - \sin^2(\mu) - 2a\kappa^{-1}\mu\sin(\mu)\cos(\mu) + (a\kappa^{-1})^2\mu^2\sin^2(\mu)\\
        &= 1 + [-1 + (a\kappa^{-1})^2\mu^2]\sin^2(\mu) - 2a\kappa^{-1}\mu\sin(\mu)\cos(\mu)\\
        &= 1 - \mu\sin(\mu)\Big[[1 - (a\kappa^{-1})^2\mu^2]\frac{\sin(\mu)}{\mu} + 2a\kappa^{-1}\cos(\mu)\Big]\\
        &= 1 - \mu\sin(\mu)\cD_0(\lambda).
    \end{align}
\end{proof}

Next, using identities \eqref{D1free} and \eqref{D0free}, we represent the Dirac points in terms of $\cos(\mu)$ and $\sin(\mu)/\mu$.

\begin{prop}
		\label{cor:DiracPointsFree}
		The set of Dirac points of $\cA_0$ is given by 
		\begin{equation}
		\big\{(\theta_1,\theta_2,\lambda) \in \bR^3 ~:~ (\theta_1,\theta_2) = \pm (2\pi/3,-2\pi/3) \quad \text{ and }\quad \Delta(\lambda;\kappa,m,a) = 0 \big\}.
		\end{equation}
		In other words, the dispersion surface of $\cA_0$ has conical singularities at any $\lambda$ satisfying
		\begin{equation}\label{CorDispD1D2free}
		\Big[ 1 - 2\frac{m}{3}\frac{\lambda}{a}(a\kappa^{-1}) \Big] \cos\big(\sqrt{\lambda/a}\big) + \Big[ a\kappa^{-1} + \frac{m}{3}\frac{\lambda}{a} - \frac{m}{3}\big(\frac{\lambda}{a}\big)^2(a\kappa^{-1})^2 \Big]\frac{\sin(\sqrt{\lambda/a})}{\sqrt{\lambda/a}} = 0.
		\end{equation}
	\end{prop}

    Next, we focus on the endpoints of the spectral bands. From the dispersion relation we know that for such points we have 
    \begin{equation}\label{DispFree}
        \Delta(\lambda;\kappa,m,a) = \pm1 \quad \Longrightarrow \quad \cD_1(\lambda) - \frac{m}{3}\frac{\lambda}{a}\cD_0(\lambda) = \pm1.
    \end{equation}
    
    \begin{remark}\label{EigenvaluesBandenpointsFree}
    If $\lambda \in \Sigma_0$, then $\cD_0(\lambda) = 0$. Therefore, $[\cD_1(\lambda)]^2 = 1$ by Lemma \ref{D1D0identity}. However, this is exactly condition \eqref{DispFree}, so $\lambda$ is an endpoint of a spectral band of $\cA_0$. Therefore, we observed that any eigenvalue of $\cA_0$ is an embedded eigenvalue of $\cA_0$ and a boundary point of the spectrum of $\cA_0$.
    \end{remark}
    
    Finally, we characterize the endpoints of the spectral bands of $\cA_0$.
        \begin{lem}
		\label{SigmaEtaD1free}
		The set of endpoints of the spectral bands of $\cA_0$ is given by the set of $\lambda \in \mathbb{R}$ satisfying at least one of the following identities. 
        \begin{enumerate}
            \item $\lambda = 0$, 
            \item $\displaystyle\frac{\sin(\sqrt{\lambda/a})}{\sqrt{\lambda/a}} = 0$ and $m\kappa^{-1}=0$,
            \item $\displaystyle\frac{\sin(\sqrt{\lambda/a})}{\sqrt{\lambda/a}} = 0$ and $\displaystyle\frac{3}{\pi m a \kappa^{-1}} \in \mathbb{N}$,
            \item $\cD_0(\lambda) = 0$, that is 
            \begin{equation*}
                \Big[1 - (a\kappa^{-1})^2\displaystyle\frac{\lambda}{a}\Big]\displaystyle\frac{\sin(\sqrt{\lambda/a})}{\sqrt{\lambda/a}} + 2a\kappa^{-1}\cos(\sqrt{\lambda/a}) = 0,
            \end{equation*}
            \item $\displaystyle\frac{m}{3}\big[\cD_1(\lambda) - 1\big] + \frac{\sin(\sqrt{\lambda/a})}{\sqrt{\lambda/a}} = 0$ that is 
            \begin{equation*}
                \frac{m}{3}\cos(\sqrt{\lambda/a}) + \Big[1 - \frac{m}{3}\frac{\lambda}{a}a\kappa^{-1}\Big]\frac{\sin(\sqrt{\lambda/a})}{\sqrt{\lambda/a}} = \frac{m}{3},
            \end{equation*}
            \item $\displaystyle\frac{m}{3}\big[\cD_1(\lambda) + 1\big] + \frac{\sin(\sqrt{\lambda/a})}{\sqrt{\lambda/a}} = 0$ that is 
            \begin{equation*}
                \frac{m}{3}\cos(\sqrt{\lambda/a}) + \Big[1 - \frac{m}{3}\frac{\lambda}{a}a\kappa^{-1}\Big]\frac{\sin(\sqrt{\lambda/a})}{\sqrt{\lambda/a}} = -\frac{m}{3}.
            \end{equation*}
        \end{enumerate}
	\end{lem} 
    \begin{proof}
        Recall that $\mu$ denotes $\sqrt{\lambda/a}$. Then, from \eqref{DispFree} we know that $\lambda \in \mathbb{R}$ is an endpoint of a spectral band if and only if
        \begin{equation}\label{dispersionD0D1proof}
            \cD_1(\lambda) - \frac{m}{3}\frac{\lambda}{a}\cD_0(\lambda) = \pm1,
        \end{equation}
        equivalently
        \begin{equation}
            \Big[\cD_1(\lambda) - \frac{m}{3}\frac{\lambda}{a}\cD_0(\lambda)\Big]^2 = 1 \quad \Longrightarrow  \quad [\cD_1(\lambda)]^2 - 2\frac{m}{3}\frac{\lambda}{a}\cD_0(\lambda)\cD_1(\lambda) + \Big(\frac{m}{3}\Big)^2\Big(\frac{\lambda}{a}\Big)^2[\cD_0(\lambda)]^2 = 1.
        \end{equation}
        Therefore, by Lemma \ref{D1D0identity} we get
        \begin{align}
            1 &= [\cD_1(\lambda)]^2 - 2\frac{m}{3}\frac{\lambda}{a}\cD_0(\lambda)\cD_1(\lambda) + \Big(\frac{m}{3}\Big)^2\Big(\frac{\lambda}{a}\Big)^2[\cD_0(\lambda)]^2\\
            1 &= 1 - \mu\sin(\mu)\cD_0(\lambda) - 2\frac{m}{3}\mu^2\cD_0(\lambda)\cD_1(\lambda) + \frac{m^2}{9}\mu^4[\cD_0(\lambda)]^2\\
            0 &= -\mu\sin(\mu)\cD_0(\lambda) - 2\frac{m}{3}\mu^2\cD_0(\lambda)\cD_1(\lambda) + \frac{m^2}{9}\mu^4[\cD_0(\lambda)]^2 \label{D0D1lastidentity}
        \end{align}
        Now we start to consider the cases we stated in Lemma \ref{SigmaEtaD1free}.\\

        \textbf{\underline{Case (1):}} If $\lambda = 0$, that is $\mu = 0$, then \eqref{D0D1lastidentity} is satisfied.\\

        \textbf{\underline{Cases (2) and (3):}} If $\sin(\mu)/\mu = 0$, then $\cos(\mu) = \pm1$, so we have 2 subcases:
        \begin{itemize}
            \item If $\cos(\mu) = 1$, then $\mu^2 = 2\pi n$ for some $n \in \mathbb{N}$, so \eqref{dispersionD0D1proof} becomes
            \begin{equation}
                  \pm1 = 1 - \frac{m}{3} (2n)\pi 2a\kappa^{-1},
            \end{equation}
            that is
            \begin{equation}
                \frac{m}{3} (2n)\pi 2a\kappa^{-1} = 0 \qquad \Longrightarrow \qquad m = 0 \quad \text{or} \quad \kappa^{-1} = 0
            \end{equation}
            or
            \begin{equation}
                \frac{m}{3} (2n)\pi a\kappa^{-1} = 1 \qquad \Longrightarrow \qquad \frac{3}{\pi ma\kappa^{-1}} \in 2\mathbb{N}.
            \end{equation}

            \item Similarly, if $\cos(\mu) = -1$, then $\mu^2 = (2n+1)\pi$ for some $n \in \mathbb{N}$, so \eqref{dispersionD0D1proof} becomes
            \begin{equation}
                  \pm1 = 1 - \frac{m}{3} (2n+1)\pi 2a\kappa^{-1},
            \end{equation}

            \begin{equation}
                \frac{m}{3} (2n+1)\pi 2a\kappa^{-1} = 0 \qquad \Longrightarrow \qquad m = 0 \quad \text{or} \quad \kappa^{-1} = 0
            \end{equation}
            or
            \begin{equation}
                \frac{m}{3} (2n+1)\pi a\kappa^{-1} = 1 \qquad \Longrightarrow \qquad \frac{3}{\pi ma\kappa^{-1}} \in 2\mathbb{N} + 1.
            \end{equation}
        \end{itemize}
        \textbf{\underline{Cases (4), (5) and (6):}} For the remaining cases we assume $\sin(\mu)/\mu \neq 0$, so \eqref{D0D1lastidentity} becomes
        \begin{align}
            0 &= -\frac{\mu}{\sin(\mu)}\mu^2\cD_0(\lambda) \Big[\frac{\sin^2(\mu)}{\mu^2} + 2\frac{m}{3}\frac{\sin(\mu)}{\mu}\cD_1(\lambda) - \frac{m^2}{9}\underbrace{\mu\sin(\mu)\cD_0(\lambda)}_{=~1 - [\cD_1(\lambda)]^2}\Big]\\
            0 &= -\frac{\mu}{\sin(\mu)}\mu^2\cD_0(\lambda) \Big[-\frac{m}{9} + \frac{\sin^2(\mu)}{\mu^2} + 2\frac{m}{3}\frac{\sin(\mu)}{\mu}\cD_1(\lambda) + \frac{m^2}{9}[\cD_1(\lambda)]^2\Big]\\
            0 &= -\frac{\mu}{\sin(\mu)}\mu^2\cD_0(\lambda) \Big[-\Big(\frac{m}{3}\Big)^2 + \Big(\frac{m}{3}\cD_1(\lambda) + \frac{\sin(\mu)}{\mu}\Big)^2\Big] 
        \end{align}
        Then we have
        \begin{equation}\label{D0D1D1Identity}
            \cD_0(\lambda)  \Big(\frac{m}{3}\cD_1(\lambda) + \frac{\sin(\mu)}{\mu} - \frac{m}{3}\Big) \Big(\frac{m}{3}\cD_1(\lambda) + \frac{\sin(\mu)}{\mu} + \frac{m}{3}\Big) = 0
        \end{equation}
        or equivalently
        \begin{equation}
             \cD_0(\lambda) \Big(\frac{m}{3}\big[\cD_1(\lambda) - 1\big] + \frac{\sin(\sqrt{\lambda/a})}{\sqrt{\lambda/a}}\Big) \Big(\frac{m}{3}\big[\cD_1(\lambda) + 1\big] + \frac{\sin(\sqrt{\lambda/a})}{\sqrt{\lambda/a}}\Big) = 0,
        \end{equation}
        which is valid if case (4), (5) or (6) is valid. This concludes the proof.
    \end{proof}

\section{Acknowledgments}
     B.H. thanks the National Science Foundation for support through the DMS-2052519 grant.

	\bibliographystyle{abbrv}
	\bibliography{ref}

\end{document}